\documentclass[12pt,reqno]{amsart}

\usepackage[left=3.9cm,right=3.9cm]{geometry}
\usepackage{amssymb}
\usepackage{graphicx}
\usepackage{amscd}
\usepackage[pagebackref]{hyperref}
\usepackage{color}
\usepackage{tabularx}
\usepackage[table]{xcolor}
\usepackage{float}
\usepackage{graphics,amsmath,amssymb}
\usepackage{amsthm}
\usepackage{amsfonts}
\usepackage{latexsym}
\usepackage{epsf}
\usepackage{xifthen}
\usepackage{mathrsfs}
\usepackage{dsfont}
\usepackage{makecell}
\usepackage{subfig}
\usepackage{amsmath}
\allowdisplaybreaks[4]
\usepackage{listings}
\usepackage{etoolbox}
\usepackage{fancyhdr}
\usepackage{pdflscape}
\usepackage[title,toc,titletoc]{appendix}
\usepackage{enumitem}
\usepackage[noadjust]{cite}
\usepackage{tikz}
\usetikzlibrary{automata,positioning,arrows}
\usepackage{young}
\usepackage[object=vectorian]{pgfornament} %%  http://altermundus.com/pages/tkz/ornament/index.html
\usepackage{lipsum,tikz}
\usepackage{multirow}
\usepackage[OT2,T1]{fontenc}
\usepackage{mathtools}

\hypersetup{
	colorlinks=true, %set true if you want colored links
	linktoc=all, %set to all if you want both sections and subsections linked
	linkcolor=blue} %choose some color if you want links to stand out

\numberwithin{equation}{section}

\theoremstyle{theorem}
\newtheorem{theorem}{Theorem}[section]
\newtheorem*{theorem*}{Theorem}

\newtheorem{corollary}[theorem]{Corollary}
\newtheorem{lemma}[theorem]{Lemma}

\newtheorem{conjecture}[theorem]{Conjecture}

\newtheorem{innercustomgeneric}{\customgenericname}
\providecommand{\customgenericname}{}
\newcommand{\newcustomtheorem}[2]{%
	\newenvironment{#1}[1]
	{%
		\renewcommand\customgenericname{#2}%
		\renewcommand\theinnercustomgeneric{##1}%
		\innercustomgeneric
	}
	{\endinnercustomgeneric}
}
\newcustomtheorem{ctheorem}{Theorem}
\newcustomtheorem{clemma}{Lemma}

\theoremstyle{definition}

\newtheorem*{example*}{Example}
\newtheorem*{examples*}{Examples}
\newtheorem{remark}[theorem]{Remark}
\newtheorem*{remark*}{Remark}
\newtheorem*{remarks*}{Remarks}
\newtheorem*{note*}{Note}

\newtheoremstyle{named}{}{}{\itshape}{}{\bfseries}{.}{.5em}{#1\thmnote{ #3}}
\theoremstyle{named}

\makeatletter
\patchcmd{\subsection}{\bfseries}{\bfseries\boldmath}{}{}
\makeatother

\newcommand{\arxiv}[1]{\href{https://arxiv.org/abs/#1}{arXiv:#1}}
\newcommand{\doi}[1]{\href{https://dx.doi.org/#1}{doi: #1}}

\newcommand{\qbinom}[2]{{#1\brack #2}}

\title[Andrews--El Bachraoui conjecture]{Proof of the Andrews--El Bachraoui positivity conjecture}

\author[S. Chern]{Shane Chern}
\address[S. Chern]{Fakult\"at f\"ur Mathematik, Universit\"at Wien, Oskar-Morgenstern-Platz 1, Wien 1090, Austria}
\email{chenxiaohang92@gmail.com, xiaohangc92@univie.ac.at}

\author[C. Wang]{Chun Wang}
\address[C. Wang]{Department of Mathematics, Shanghai Normal University, Shanghai 200234, People's Republic of China}
\email{wangchun@shnu.edu.cn}

\date{}



\begin{document}
	
\sloppy

\begin{abstract}
	We prove that for $k\ge 1$, all coefficients in the expansion of the series
	\begin{align*}
		\sum_{n\ge 0} \frac{(q^{2n+2}, q^{2n+2k}; q^2)_\infty}{(q^{2n+1};q^2)_\infty^2} q^{2n}
	\end{align*}
	are positive, by $q$-hypergeometric means. This confirms a recent conjecture of Andrews and El Bachraoui.
\end{abstract}

\maketitle

\section{Introduction}

A \emph{power series} $F(q)=\sum_{n\ge 0} f_n q^n \in \mathbb{R}[[q]]$ is called \emph{nonnegative} if all its coefficients satisfy $f_n\ge 0$; likewise, the series is \emph{positive} if $f_n > 0$ for every index $n\ge 0$. In addition, if the series $F(q)$ becomes a \emph{polynomial} in $\mathbb{R}[q]$, then we say it is nonnegative (resp.~positive) if its coefficients are all nonnegative (resp.~positive) up to the highest power of $q$.

Nonnegative series or polynomials are usually connected with enumerative combinatorics. In particular, if the series or polynomial gives the generating function for certain combinatorial objects, the nonnegativity is guaranteed by the corresponding enumerations. A classical example on this topic is the \emph{$q$-binomial coefficient}, defined by
\begin{align*}
	\qbinom{M}{N}_q:=\begin{cases}
		\dfrac{(q;q)_M}{(q;q)_N(q;q)_{M-N}}, & \text{if $0\le N\le M$},\\[10pt]
		0, & \text{otherwise}.
	\end{cases}
\end{align*}
Here and throughout, we adopt the conventional \emph{$q$-Pochhammer symbols} for $n\in \mathbb{N}\cup\{\infty\}$:
\begin{align*}
	(A;q)_n := \prod_{k=0}^{n-1} (1-Aq^k),
\end{align*}
with the compact notation
\begin{align*}
	(A_1,A_2,\ldots,A_m;q)_n := (A_1;q)_n (A_2;q)_n \cdots (A_m;q)_n.
\end{align*}
For nonnegative integers $M\ge N\ge 0$, it is a standard result \cite[p.~33, Theorem~3.1]{And1998} that $\qbinom{M}{N}_q$ is the generating function for integer partitions into at most $N$ parts, each no larger than $M - N$. This enumerative fact reveals that $\qbinom{M}{N}_q$ is a \emph{positive polynomial} of degree $N(M-N)$.

In their recent work on a bias phenomenon for partitions into two colors, Andrews and El Bachraoui~\cite[eq.~(1.2) with $m=1$]{AEB2025} considered the following $q$-series\footnote{Here we multiply a prefactor $q^{-1}$ by the original series of Andrews and El Bachraoui so that the constant term does not vanish.}:
\begin{align}\label{eq:AEB-F}
	F_{k,1}(q) := \sum_{n\ge 0} \frac{(q^{2n+2}, q^{2n+2k}; q^2)_\infty}{(q^{2n+1};q^2)_\infty^2} q^{2n}.
\end{align}
They observed in \cite[eq.~(2.1)]{AEB2025} that this series gives the generating function for the difference of the enumerations of two disjoint subsets of a special family of bicolored partitions. Along this line, they conjectured the positivity of this series, which implies an inequality for the two partition enumerations.

\begin{conjecture}[Andrews--El Bachraoui, {\cite[Conjecture~2.8]{AEB2025}}]
	For every positive integer $k$, the series $F_{k,1}(q)$ is positive.
\end{conjecture}

This conjecture was confirmed by Andrews and El Bachraoui~\cite[eqs.~(2.2), (2.3) and (2.5), and Theorem~2.6]{AEB2025} for $1\le k\le 4$. In a subsequent paper, Banerjee, Bringmann, and Keith \cite[Theorems~1.2 and 1.3]{BBK2026} extended the study to the cases where $5\le k\le 10$, and established a weaker result that $F_{k,1}(q)$ is a nonnegative series.

The main objective of this note is to confirm this conjecture of Andrews and El Bachraoui.

\begin{theorem}\label{th:AEB}
	The Andrews--El Bachraoui positivity conjecture is true.
\end{theorem}

Our approach is entirely different from those in \cite{AEB2025} and \cite{BBK2026}, and is indeed much simpler. To be precise, it is built upon the following positivity result of independent interest.

\begin{theorem}\label{th:ratio-n}
	For every positive integer $n$, the series $(q^2;q^2)_n/(q;q^2)_n^2$ is positive.
\end{theorem}

\begin{remark}
	The limiting case with $n\to\infty$ states that all coefficients in the series expansion of the infinite product $(q^2;q^2)_\infty/(q;q^2)_\infty^2$ are positive. The only known proof of this fact, to the best of knowledge, is based on its connection with a third-order mock theta function
	\begin{align*}
		\nu(q) := \sum_{n\ge 0} \frac{q^{n(n+1)}}{(q;q^2)_{n+1}}.
	\end{align*}
	Letting $\mathbf{U}_m$ be the \emph{unitizing operator of degree $m$} defined by
	\begin{align*}
		\mathbf{U}_m\left(\sum_{n\ge 0} a_n q^n\right) := \sum_{n\ge 0} a_{mn} q^n,
	\end{align*}
	Hirschhorn and Sellers~\cite[eqs.~(2.5) and (2.8)]{HS2014} observed that
	\begin{align*}
		\mathbf{U}_2\big(\nu(-q)\big) = (q;q)_\infty (-q;q)_\infty^3 = \frac{(q^2;q^2)_\infty}{(q;q^2)_\infty^2}.
	\end{align*}
	Since all coefficients in the expansion of $\nu(-q)$ are positive, so are the even-indexed coefficients.
\end{remark}

\section{A ``substitution of one'' trick}

We start with the proof of Theorem~\ref{th:ratio-n}. This result relies on a trick which we call the ``\emph{substitution of one}.'' In what follows, we adopt the \emph{$q$-hypergeometric functions ${}_{r}\phi_s$} defined by
\begin{align*}
	{}_{r}\phi_s\left(\begin{matrix} A_1,A_2,\ldots,A_r\\ B_1,B_2,\ldots,B_s \end{matrix}; q, z\right):=\sum_{n\ge 0} \frac{(A_1,A_2,\ldots,A_r;q)_n \big((-1)^n q^{\binom{n}{2}}\big)^{s-r+1} z^n}{(q,B_1,B_2,\ldots,B_{s};q)_n}.
\end{align*}

\begin{lemma}
	For every nonnegative integer $m$,
	\begin{align}\label{eq:1-sub}
		\sum_{i=0}^m \frac{(q;q^2)_i (q;q^2)_{m-i}}{(q^2;q^2)_i (q^2;q^2)_{m-i}} q^i = 1.
	\end{align}
\end{lemma}

\begin{proof}
	We first rewrite the left-hand side of \eqref{eq:1-sub} in terms of a ${}_{2}\phi_1$ series:
	\begin{align*}
		\sum_{i=0}^m \frac{(q;q^2)_i (q;q^2)_{m-i}}{(q^2;q^2)_i (q^2;q^2)_{m-i}} q^i &= \frac{(q;q^2)_{m}}{(q^2;q^2)_{m}} \sum_{i=0}^m \frac{(q,q^{-2m};q^2)_i q^{2i}}{(q^2,q^{1-2m};q^2)_i}\\
		&= \frac{(q;q^2)_{m}}{(q^2;q^2)_{m}} {}_{2}\phi_1\left(\begin{matrix} q,q^{-2m}\\ q^{1-2m} \end{matrix}; q^2, q^2\right),
	\end{align*}
	where we have applied the relation
	\begin{align*}
		\frac{(q;q^2)_{m-i}}{(q^2;q^2)_{m-i}} = \frac{(q;q^2)_{m}}{(q^2;q^2)_{m}} \frac{(q^{-2m};q^2)_i}{(q^{1-2m};q^2)_i} q^i.
	\end{align*}
	Now applying the \emph{second Chu--Vandermonde sum}~\cite[p.~424, eq.~(17.6.3)]{And2010}:
	\begin{align*}
		{}_{2}\phi_1\left(\begin{matrix} a,q^{-m}\\ c \end{matrix}; q, q\right) = \frac{a^m (c/a;q)_m}{(c;q)_m}
	\end{align*}
	with $(a,c,q)\mapsto (q,q^{1-2m},q^2)$, we have
	\begin{align*}
		{}_{2}\phi_1\left(\begin{matrix} q,q^{-2m}\\ q^{1-2m} \end{matrix}; q^2, q^2\right) = \frac{q^m (q^{-2m};q^2)_m}{(q^{1-2m};q^2)_m} = \frac{(q^2;q^2)_{m}}{(q;q^2)_{m}}.
	\end{align*}
	The claimed relation therefore follows.
\end{proof}

Now we are ready to prove Theorem~\ref{th:ratio-n}.

\begin{proof}[Proof of Theorem~\ref{th:ratio-n}]
	In light of \eqref{eq:1-sub},
	\begin{align*}
		\frac{(q^2;q^2)_n}{(q;q^2)_n^2} &= \frac{(q^2;q^2)_n}{(q;q^2)_n^2} \sum_{i=0}^n \frac{(q;q^2)_i (q;q^2)_{n-i}}{(q^2;q^2)_i (q^2;q^2)_{n-i}} q^i\\
		&= \sum_{i=0}^n \qbinom{n}{i}_{q^2} \frac{q^i}{(q^{2i+1};q^2)_{n-i} (q^{2(n-i)+1};q^2)_i}.
	\end{align*}
	All summands in the above are nonnegative series, and in particular, the term at $i=0$ is $1/(q;q^2)_n$, which is a positive series for any $n\ge 1$. We hence arrive at the desired result.
\end{proof}

We record a corollary of Theorem~\ref{th:ratio-n} for our later use.

\begin{corollary}\label{coro:ratio-n+1}
	For every nonnegative integer $n$, the series $(q^2;q^2)_n/(q;q^2)_{n+1}^2$ is positive.
\end{corollary}

\begin{proof}
	When $n=0$, the series in question is $1/(1-q)^2$, which is clearly positive. When $n\ge 1$, we rewrite it as
	\begin{align*}
		\frac{(q^2;q^2)_n}{(q;q^2)_{n+1}^2} = \frac{1}{(1-q^{2n+1})^2}\cdot\frac{(q^2;q^2)_n}{(q;q^2)_n^2},
	\end{align*}
	and then use Theorem~\ref{th:ratio-n}.
\end{proof}

\section{The Andrews--El Bachraoui positivity conjecture}

To work on the Andrews--El Bachraoui positivity conjecture, we require a reformulation of the original series $F_{k,1}(q)$, which is due to Banerjee, Bringmann, and Keith \cite[Lemma~3.2]{BBK2026}:
\begin{align}\label{eq:BBK-reform}
	F_{k,1}(q) = \sum_{n\ge 0} \frac{(q^{2k-1};q^2)_n q^{n}}{(q;q^2)_{n+1}}.
\end{align}
For its proof, one simply rewrites $F_{k,1}(q)$ in \eqref{eq:AEB-F} as
\begin{align*}
	F_{k,1}(q) = \frac{(q^2;q^2)_\infty^2}{(q;q^2)_\infty^2 (q^2;q^2)_{k-1}} {}_{2}\phi_1\left(\begin{matrix} q,q\\ q^{2k} \end{matrix}; q^2, q^2\right),
\end{align*}
and then applies \emph{Heine's first tranformation}~\cite[p.~424, eq.~(17.6.6)]{And2010}:
\begin{align*}
	{}_{2}\phi_1\left(\begin{matrix} a,b\\ c \end{matrix}; q, z\right) = \frac{(b,az;q)_\infty}{(c,z;q)_\infty} {}_{2}\phi_1\left(\begin{matrix} c/b,z\\ az \end{matrix}; q, b\right)
\end{align*}
with $(a,b,c,z,q)\mapsto (q,q,q^{2k},q^2,q^2)$.

Taking $k=1$ in \eqref{eq:BBK-reform} gives
\begin{align*}
	F_{1,1}(q) = \sum_{n\ge 0} \frac{q^{n}}{1-q^{2n+1}},
\end{align*}
which is \cite[eq.~(2.2)]{AEB2025} and implies the positivity immediately.

For the moment, we assume $k\ge 2$. To confirm the positivity conjecture, we further need to reformulate \eqref{eq:BBK-reform} in the following way.

\begin{lemma}
	For $k\ge 2$,
	\begin{align}\label{eq:CW-reform}
		F_{k,1}(q) = \sum_{n=0}^{k-2} \qbinom{k-2}{n}_{q^2} \frac{(q^2;q^2)_n q^{2n^2+2n}}{(q;q^2)_{n+1}^2}.
	\end{align}
\end{lemma}

\begin{proof}
	Recall from \eqref{eq:BBK-reform},
	\begin{align*}
		F_{k,1}(q) = \frac{1}{1-q}\cdot {}_{2}\phi_1\left(\begin{matrix} q^2, q^{2k-1}\\ q^3 \end{matrix}; q^2, q\right).
	\end{align*}
	Now all we need is \emph{Fine's second transformation}~\cite[p.~424, eq.~(17.6.10)]{And2010}:
	\begin{align*}
		{}_{2}\phi_1\left(\begin{matrix} q, aq\\ bq \end{matrix}; q, z\right) = \frac{1}{1-z}\sum_{n\ge 0} \frac{(b/a;q)_n (-az)^n q^{(n^2+n)/2}}{(bq,zq;q)_n}.
	\end{align*}
	With the substitution $(a,b,z,q)\mapsto (q^{2k-3},q,q,q^2)$, we have
	\begin{align*}
		F_{k,1}(q) = \frac{1}{1-q}\cdot \frac{1}{1-q} \sum_{n=0}^{k-2} \frac{(q^{4-2k};q^2)_n (-q^{2k-2})^n q^{n^2+n}}{(q^3;q^2)_n^2}.
	\end{align*}
	Applying the relation
	\begin{align*}
		(q^{4-2k};q^2)_n &= (-1)^n q^{(4-2k)n+n^2-n} (q^{2k-4};q^{-2})_n\\
		&= (-1)^n q^{(4-2k)n+n^2-n} (q^2;q^2)_n \qbinom{k-2}{n}_{q^2},
	\end{align*}
	we finally obtain the claimed identity.
\end{proof}

Now the positivity of $F_{k,1}(q)$ is a direct consequence of \eqref{eq:CW-reform}.

\begin{proof}[Proof of Theorem~\ref{th:AEB}]
	It remains to consider the cases where $k\ge 2$. In \eqref{eq:CW-reform}, it is apparent that all summands
	\begin{align*}
		\qbinom{k-2}{n}_{q^2} \frac{(q^2;q^2)_n q^{2n^2+2n}}{(q;q^2)_{n+1}^2}
	\end{align*}
	are nonnegative series because of Corollary~\ref{coro:ratio-n+1}. In particular, the term at $n=0$ is $1/(1-q)^2$, which is further a positive series. The positivity of the Andrews--El Bachraoui series $F_{k,1}(q)$ therefore holds.
\end{proof}

\subsection*{Acknowledgements}

Shane Chern was supported by the Austrian Science Fund (Grant No.~10.55776/F1002). Chun Wang was partially supported by the
National Natural Science Foundation of China
(Grant No.~12571354), Shanghai Rising--Star Program (Grant No.~23QA1407300), and China Scholarship Council (Grant No.~202508310241).

\bibliographystyle{amsplain}

\end{document}